\documentclass[a4paper,10pt]{amsart}
\usepackage[utf8]{inputenc}
\usepackage{amsmath,amssymb,amsthm}
\usepackage{xcolor}
\usepackage{enumerate}
\usepackage{tikz}
\usepackage{tikz-cd}
\usepackage[all]{xy}
\usepackage{float}
\usepackage{soul}
\usepackage{thmtools}

\usepackage{listings}

\usepackage{hyperref}
\hypersetup{hidelinks}

\usepackage[noabbrev,capitalise,nameinlink]{cleveref}
\crefname{subsection}{Subsection}{Subsections}

\usepackage{stmaryrd}

\theoremstyle{plain}

\newtheorem{thm}{Theorem}[section]
\newtheorem{corollary}[thm]{Corollary}
\newtheorem{lemma}[thm]{Lemma}
\newtheorem{proposition}[thm]{Proposition}

\newtheorem{thmABC}{Theorem}

\newtheorem{corABC}[thmABC]{Corollary}

\newtheorem*{conjecture*}{Conjecture}

\theoremstyle{definition}

\newtheorem{definition}[thm]{Definition}
\newtheorem{remark}[thm]{Remark}
\newtheorem{example}[thm]{Example}
\newtheorem{prob}[thm]{Problem}

\newcommand{\bbQ}{\mathbb{Q}}
\newcommand{\bbF}{\mathbb{F}}
\newcommand{\bbZ}{\mathbb{Z}}
\newcommand{\bbN}{\mathbb{N}}

\renewcommand{\epsilon}{\varepsilon}

\numberwithin{equation}{section}

\newcommand{\comm}[1]{}

\title[On probabilistic identities and coset identities in pro-$p$ groups]{On probabilistic identities and \\ coset identities in pro-$p$ groups}
\author[Kionke]{Steffen Kionke}
\address{FernUniversit\"at in Hagen, Fakult\"at f\"ur Mathematik und Informatik, 58084 Hagen, Germany}
\email{steffen.kionke@fernuni-hagen.de}
\author[Otmen]{Nowras Otmen}
\address{Dipartimento di Matematica `Tullio Levi-Civita', Università degli Studi di Padova, Via Trieste, 63 - 35131 Padova, Italy}
\email{nowras.naufel@math.unipd.it}
\author[Toti]{Tommaso Toti}
\address{Dipartimento di Matematica e Applicazioni, Università degli Studi di Milano-Bicocca, Via Roberto Cozzi, 55 - 20126, Milano, Italy}
\email{t.toti@campus.unimib.it}
\address{Departamento de Matem\'aticas, Universidad del Pa\'is Vasco/Euskal Herriko Unibertsitatea, Apartado 644 - 48080, Bilbao, Spain}
\email{ttoti001@ikasle.ehu.eus}
\author[Vannacci]{Matteo Vannacci}
\address{Dipartimento di Matematica `Ulisse Dini', Universit\`a degli Studi di Firenze, Viale Morgagni, 67/a - 50134 Firenze, Italy}
\email{matteo.vannacci@unifi.it}
\author[Weigel]{Thomas Weigel}
\address{Dipartimento di Matematica e Applicazioni, Università degli Studi di Milano-Bicocca, Via Roberto Cozzi, 55 - 20126, Milano, Italy}
\email{thomas.weigel@unimib.it}

\subjclass[2010]{Primary 20P05;
Secondary 20E18}

\date{\today}

\begin{document}

\begin{abstract}
It is shown that a probabilistic identity on a $\sigma$-compact $K$-analytic group $G$, $K$ a non-archimedean local field, is a coset identity. As an application, one concludes that compact $K$-analytic groups and various pro-$p$ groups obtained from free constructions satisfy a probabilistic Tits alternative. By means of Lie-theoretic methods, we also study torsion probabilistic identities in virtually free pro-$p$ and compact $p$-adic analytic groups.
\end{abstract}

\maketitle

\section{Introduction}

\subsection{State of the art}
The main purpose of this paper is to study the interplay between three important concepts in Group Theory, identities, coset identities and probabilistic identities, in certain classes of pro-$p$ groups.

Let $\Gamma$ be a group and let $w\in F_r$ be a non-trivial word in the free abstract group of rank $r$. One says that $w$ is an identity on $\Gamma$ if for all $g_1,\ldots,g_r \in \Gamma$ one has $w(g_1,\ldots,g_r)=1$.
The study of identities in groups has a long tradition, e.g., already W.~Burnside
raised the question whether a finitely generated group of finite exponent $n$ is necessarily a finite group. Several interesting classes of groups are defined by imposing an identity: abelian, nilpotent and solvable groups, and groups of finite exponent. Moreover, the restricted Burnside problem, which was one of the central problems in Algebra of the twentieth century, was centred around torsion identities.

A natural generalisation of identities is the concept of \emph{coset identities}. One says that $w$ is a coset identity on $\Gamma$  if there exists a finite index subgroup $H$ of $\Gamma$ and elements $g_1,\ldots,g_r \in \Gamma$ such that $w(g_1h_1,\ldots,g_rh_r)=1$ for all $h_1,\dots,h_r\in H$. Even though it is clear that any identity is a coset identity, all known examples of groups satisfying a coset identity also satisfy an identity.

In recent years probabilistic methods in Group Theory have been  used quite successfully. Some difficult group-theoretical results - like the Magnus problem - have been reproved with a much shorter probabilistic proof (see \cite{DPSS}).  One natural way to introduce probability on a profinite group $G$ is to endow the group $G$ with the Haar probability measure $\mu$ making $(G,\mathcal{A},\mu)$ a probability space, where
$\mathcal{A}$ denotes the Borel $\sigma$-algebra.

If $\Gamma$ is residually finite, one may talk about the probability $P(\Gamma,w)$ that $k$ random elements in the profinite completion $\widehat{\Gamma}$ satisfy $w$. One  says that $w$ is a \emph{probabilistic identity} for $\Gamma$ if $P(\Gamma,w)$ is positive. If $\Gamma$ does not admit any probabilistic identity, one says that it is \emph{randomly free}. Again, it is clear that a coset identity yields a probabilistic identity. However, to the best of the authors' knowledge, all known examples of profinite groups which admit probabilistic identities are indeed profinite groups which admit identities. 

The interplay between identities, coset identities and probabilistic identities has been investigated 
quite intensively.
\comm{several times in the literature.} 
In 2018 Larsen and Shalev proved that groups with infinitely many non-isomorphic non-abelian upper composition factors are randomly free \cite[Thm.~1.5]{LS18}; thus, one may restrict their attention to groups with restricted composition factors, e.g.\ pro-$p$ groups. Their result also implies that finitely generated linear groups satisfy a \emph{probabilistic Tits alternative}: they are either virtually solvable or randomly free. Moreover, by \cite[Theorem 1.4]{LS16} the concepts of coset identity and probabilistic identity coincide for groups in this class. Other works investigating the relation between identities and coset identities include \cite{WZ}, \cite{CM07}, \cite{BG} \cite{Ni} and \cite{LP00}.

Already in 2000 Aner Shalev made the following ``very challenging conjectures'' (loc.\ cit.):

\begin{conjecture*}[{\cite[Conjectures 2 and 3, pg.~7]{Sh00}}] \qquad

    \begin{enumerate}
        \item[(a)]\label{conj:a} Let $G$ be a finitely generated pro-$p$ group satisfying some probabilistic identity $w$. Then $G$ satisfies some identity.
        \item[(b)]\label{conj:b} Let $G$ be a finitely generated pro-$p$ group satisfying some coset identity. Then $G$ satisfies some identity.
    \end{enumerate}
\end{conjecture*}

Conjecture~(b) was answered affirmatively for $p$-adic analytic pro-$p$ groups in \cite[Prop.~1.9]{BG}.

In this article we study two main classes of finitely generated pro-$p$ groups, $p$-adic analytic pro-$p$ groups and certain ``free constructions'' of pro-$p$ groups (e.g.\ free pro-$p$ products), and we show that groups in these classes satisfy a probabilistic Tits alternative. In particular, this answers Conjecture (a) above affirmatively for these groups.

\subsection{Main results}
We say that a group word $w$ is an \emph{open coset identity} in a topological group $G$, if $w$ is satisfied on a product of cosets of some open subgroup. In profinite groups open coset identities are coset identities as defined above. Our first main result shows that for $\sigma$-compact groups which are analytic over a non-archimedean local field $K$ the concepts of satisfying a probabilistic identity and satisfying a open coset identity coincide.

\begin{thmABC}\label{thm:analytic}
Let $K$ be a non-archimedean local field and let $G$ be a $\sigma$-compact $K$-analytic group. If $w$ is a probabilistic identity for $G$, then $w$ is an open coset identity.
\end{thmABC}

One of the most studied classes of finitely generated pro-$p$ groups is that of $K$-analytic pro-$p$ groups, for $K$ a non-archimedean local field; for instance, closed pro-$p$ subgroups of $\mathrm{GL}_n(\mathbb{Z}_p)$ and $\mathrm{GL}_n(\mathbb{F}_p[\![t]\!])$ are such examples. Here we show that a probabilistic Tits alternative holds for this class. 
By \cite[Thm.~1.3]{BG} and \cite[Proposition 5.1]{JZ02}, the following corollary immediately follows. We thank Andoni Zozaya for pointing out to us Jaikin--Zapirain's result.

\begin{corABC}\label{cor:3.5}
    Let $G$ be a compact $K$-analytic group for a local field $K$.
    Then $G$ is either virtually solvable or randomly free. In particular, a $p$-adic analytic pro-$p$ group is either solvable or it is randomly free.
\end{corABC}

The proof of \cref{thm:analytic} is of analytic flavour and relies on the Weierstra\ss{} Preparation Theorem. Using \cref{cor:3.5} we can also give a new proof of \cite[Thm.~1.4]{LS16}, showing that a linear group satisfying a probabilistic identity must also satisfy a coset identity; see \cref{cor:general_LS_1.4}. Note that our argument also applies to probabilistic identities with parameters.

We now concentrate on randomly free groups. One of the first general methods to obtain randomly free groups appears in \cite{Ab05}, where it is proved that weakly branch groups are randomly free by analysing their actions on trees. In \cite{Sz05}, it is shown that also the Nottingham group $\mathcal{N}_p$ is randomly free by a different method. Using these results, it follows that free pro-$p$ groups are also randomly free for every odd prime and for $p=2$ and rank greater than 2. Interestingly, we do not know of a direct proof of this fact. As far as the authors are aware, the assertion that the free pro-$2$ group of rank $2$ is randomly free does not follow from previous results and requires \cref{vfpprobtits} below.

While it is clear that a group with a randomly free quotient is randomly free, it is not at all clear if a group which has a finite-index randomly free subgroup is itself randomly free. In \cref{sec:free_constr}, we prove a particular instance of that; namely, that virtually non-abelian free pro-$p$ groups are randomly free. Slightly more general, an application of  
 \cref{cor:3.5} gives:

\begin{thmABC}\label{vfpprobtits}
Let $G$ be a large pro-$p$ group. Then $G$ is randomly free.
\end{thmABC} 

Recall that a pro-$p$ group is \emph{large}, if some open subgroup projects onto a non-abelian free pro-$p$ group. \cref{vfpprobtits} allows us to deduce that non-solvable Demushkin pro-$p$ groups and that several pro-$p$ groups obtained as fundamental groups of finite graphs of pro-$p$ groups, for example various amalgamated free pro-$p$ products, pro-$p$ HNN-extensions and pro-$p$ analogues of limit groups, are randomly free. For some of the aforementioned classes of pro-$p$ groups, we can also prove that they satisfy a probabilistic Tits alternative as stated in the following theorem.

\begin{thmABC}\label{thm:list}
    Let $\mathcal{G}$ one of the following classes of pro-$p$ groups:
    \begin{enumerate}[(i)]
        \item finitely generated Demushkin groups,
        \item non-trivial free products of pro-$p$ groups,
        \item pro-$p$ analogues of limit groups (in the sense of \cite{KZ}),
    \end{enumerate}
    and let $G$ be a group in $\mathcal{G}$. Then $G$ is either solvable or randomly free.
\end{thmABC}
Note that there are examples of pro-$p$ groups for which a probabilistic Tits alternative does not hold; see \cite[Corollary 10]{CM07} for an example of a non-solvable finitely generated pro-$p$ group which satisfies an identity.

 It seems worth noting that virtually free pro-$p$ groups and the groups considered in Theorem~\ref{thm:list} have positive rank gradient. It would be interesting to know whether all such pro-$p$ groups are randomly free (see \cref{prob:rank_grad}). We remark that a finitely generated pro-$p$ group of positive rank gradient cannot satisfy a coset identity by \cite[Thm.~1.7]{Ni}. See the end of \cref{sec:free_constr} for a detailed discussion.

We want to emphasise that torsion identities hold a special place in the world of identities due to historical reasons (cfr.\ Burside problem). We remark also that \cref{thm:analytic} answers \cite[Conjecture 3]{LP00} positively for compact subgroups of $\mathrm{GL}_n(K)$ where $K$ is a local field. It is therefore worth having other techniques to deal with torsion identities in other groups. In \cref{sec:torsion} we use Lie-theoretic methods to give a different proof of some of the previous results for torsion identities. In \cref{subsec:virt_free_torsion}, we study the measure of conjugacy classes in profinite groups and we show that the set of torsion elements has measure zero in virtually free pro-$p$ groups. We believe that the technique developed there can be applied to other classes of groups. 

Finally, in \cref{subsec:p-adic_torsion}, we study torsion identities in compact $p$-adic analytic groups. In \cite{FR25}, the authors ask if finitely generated pro-$p$ groups which are not $p$-adic analytic can have the set of torsion of elements of positive measure. Here, we tackle the `dual' question; namely, using \cref{thm:analytic} and Lie-theoretic methods, we obtain necessary and sufficient conditions for a $p$-adic analytic pro-$p$ group $G$ to have the set of torsion elements of positive measure.
\begin{corABC}\label{cor:p-adic_pro-p_torsion}
    Let $G$ be a $p$-adic analytic pro-$p$ group and $\mu$ its normalised Haar measure. Then $\mu(\mathrm{Tor}(G)) > 0$ if, and only if, $G$ is solvable and there exists an open uniform pro-$p$ subgroup $U$ of $G$ and $t \in \mathrm{Tor}(G)$ such that $C_U(t)=1$.
\end{corABC}
This is a direct consequence of the more general \cref{torsionpadic}. Moreover, even though having torsion of positive measure can only occur in the solvable case due to \cref{cor:3.5}, in \cref{cor:krekninpadic} we give another proof of it. In an ongoing work, we consider the problem of classifying which solvable $p$-adic analytic pro-$p$ groups can have torsion of positive measure.

\subsection*{Notational conventions}

 As it is customary when working with profinite groups, all subgroups will be assumed to be closed and all homomorphisms will be assumed to be continuous, unless stated otherwise. For a group $G$, we write $\mathrm{Tor}(G)$ to denote its set of torsion elements.

\subsection*{Acknowledgements} 
S.K.\ acknowledges funding by the Deutsche Forschungsgemeinschaft (DFG, German
Research Foundation) – 541392601. N.O.\ thanks Henrique Souza and Julian Wykowski for helpful discussions on $p$-adic analytic groups and profinite Bass-Serre theory. T.T. and T.W.\ are members of the Gruppo Nazionale per le Strutture Algebriche, Geometriche e le loro Applicazioni (GNSAGA) of the Istituto Nazionale di Alta Matematica (INdAM). 
T.W. gratefully acknowledges financial support by
the PRIN2022 “Group theory and its applications”. M.V.\ is funded by the Italian program Rita Levi Montalcini for young researchers, Edition 2021. T.T.\ and M.V.\ were also supported by the Spanish Government grant PID2020-117281GB-I00, partly by the European Regional Development Fund (ERDF), the MICIU /AEI /10.13039/501100011033 / UE grant
PCI2024-155053-2. This work was partially supported by the Basque Government grant IT1483-22.

\section{Preliminaries}

In this section, we recall the necessary definitions, notations and basic results that we will need throughout the paper. 

\subsection{Probabilistic identities} 
A profinite group $G$, being a compact Hausdorff topological group, admits a unique Haar measure $\mu_G$ defined on the $\sigma$-algebra generated by open subsets of $G$ and such that $\mu_G(G)=1$. Let $w = w(x_1,\ldots,x_r)$ be a non-trivial word in the free group of rank $r$ and define, for a profinite group $G$, the associated word map 
\[
\begin{matrix}
	w_G \colon &  G^r  & \longrightarrow & G\\
	 & (g_1,\dots, g_r) & \longmapsto & w(g_1,\dots, g_r).
\end{matrix}
\]
Since $w_G$ is continuous, the subset \[ X(G,w)=\{(g_1,\dots,g_r)\in G^r\mid w(g_1,\dots,g_r)=1\} \]
is closed in $G^r$, hence measurable, and we call $P(G,w):=\mu_{G^r}(X(G,w))$ the probability that $G$ satisfies $w$. 
\begin{definition}
	Let $G$ be a profinite group and let $w$ be a non-trivial word.  We say that $w$ is a \emph{probabilistic identity} on $G$ if $P(G,w)>0$. We say that $G$ is \emph{randomly free} if it does not satisfy any probabilistic identity.
\end{definition}

We remark that the definition of randomly free given above is equivalent to saying that, for every $n\geq 1$, the probability that an $n$-tuple of elements of $G$ generates an abstract free subgroup of rank $n$ is 1 (\cite[Lemma 1.4]{LS18}).

Further, we say that $w = w(x_1,\ldots,x_r)$ is a \emph{coset identity} if there exists $H \leq_o G$ and elements $g_1,\ldots,g_r \in G$ such that $w(g_1 h_1,\ldots,g_r h_r) = 1$ for any $h_1,\ldots,h_r \in H$. Note that every coset identity is a probabilistic identity, since $g_1H \times \ldots \times g_rH \subseteq X(G,w)$ implies $P(G,w) \geq 1/[G:H]^r>0$. 

Note that $w=w(x_1,\dots,x_r)$ is a coset identity on a residually finite group $\Gamma$ if and only if it is a coset identity on its profinite completion $\widehat{\Gamma}$. Indeed, if $H$ is a finite index subgroup of $\Gamma$ and $g_1,\dots,g_r\in\Gamma$ such that $w(g_1h_1,\dots,g_rh_r)=1$ for every $h_1,\dots,h_r\in H$, then the closure $\overline{H}$ of $H$ in $\widehat{\Gamma}$ is an open subgroup of $\widehat{\Gamma}$ and $w(g_1\bar{h}_1,\dots,g_r\bar{h}_r)=1$ for every $\bar{h}_1,\dots,\bar{h}_r\in\overline{H}$. Conversely, let $H\leq_o\widehat{\Gamma}$ and let $g_1,\ldots,g_r \in G$ such that $w(g_1 h_1,\ldots,g_r h_r) = 1$ for any $h_1,\ldots,h_r \in H$. Since $\Gamma$ is dense in its profinite completion, $g_iH\cap\Gamma$ is non-empty and open in the profinite topology of $\Gamma$ for every $i=1,\dots,r$. Therefore, there is $N\leq \Gamma$ of finite index and $\gamma_i\in\Gamma$ such that $\gamma_iN\subseteq g_iH$ so that $w$ is a coset identity on $\Gamma$.   

We collect in the following lemma some useful properties of the Haar measure for profinite groups. Note that \cref{eq:leqquotient} was already known for countably based profinite groups, see for example \cite[Lemma 11.1.1]{LS03}.

\begin{lemma}\label{quotient}
	   	Let $G$ be a profinite group and let $X$ be a closed subset of $G$. If $K\unlhd G$ and $\pi\colon G\to G/K$ is the natural projection, then \begin{equation}\label{eq:leqquotient}
	   		\mu_G(X)\leq\mu_{G/K}(\pi(X)).
	   	\end{equation}
	   	 If moreover $G$ is countably based, then 
         \begin{equation}\label{eq:measurefinitequotients}
     	\mu_G(X)=\underset{N\unlhd_oG}{\mathrm{inf}}\frac{|XN/N|}{|G/N|}=\underset{i\geq 1}{\mathrm{inf}}\frac{|XN_i/N_i|}{|G/N_i|},
     \end{equation}
     for every $\{N_i\}_{i\geq 1}$ descending chain of open normal subgroups of $G$ with trivial intersection.
\end{lemma}
\begin{proof}
	It is well known, for example by \cite[Section 63, Theorem C]{H50}, that the normalized Haar measure of the quotient group $G/K$ is the pushforward of $\mu_G$ by $\pi$, that is $\mu_{G/K}(E)=\mu_G(\pi^{-1}(E))$ for every Borel subset $E\subseteq G/K$. If $X\subseteq G$ is closed, it is in particular compact and so $\pi(X)$ is a closed hence measurable subset of $G/K$. So we have
	\[\mu_{G/K}(\pi(X)) = \mu_G(\pi^{-1}(\pi(X))) \geq \mu(X),\]
	since $X \subseteq \pi^{-1}(\pi(X))$. \cref{eq:measurefinitequotients} follows from \cite[p.~206]{LS03}.
\end{proof}
If $w$ is a non-trivial word, since $\pi (X(G,w))\subseteq X(G/K,w)$, \cref{eq:leqquotient} implies that $P(G,w)\leq P(G/K,w)$. This observation immediately implies the following corollary, which will be frequently used in the rest of the article. 

\begin{corollary}\label{cor:quot_rand_free}
A profinite group that projects onto a randomly free group is itself randomly free.     
\end{corollary}

\subsection{\texorpdfstring{$K$}{K}-analytic groups}

Let $K$ be a non-archimedean local field, with norm $\lvert \cdot \rvert$, and let $R$ be its ring of integers. Then $R$ is a compact discrete valuation ring and we fix an uniformiser $\pi \in R$ that generates its maximal ideal $\mathfrak m$. Note that, for $x \in K\smallsetminus \{0\}$, then $|x| = d^{k}$, where $d \in (0,1)$ and $k$ is the integer such that $x \in \pi^{k}R\smallsetminus \pi^{k+1}R$.

We follow the exposition in \cite[Pt.~2]{Ser65}. For $r \in \mathbb R^n_{>0}$ and $x \in K^n$, by $|x| \leq r$ we mean that $|x_i|\leq r_i$ for each $1\leq i \leq n$. For an open subset $U$ of $K^{m}$, we say that a function $f : U \to K^{n}$ is \emph{analytic} if, for each $x \in U$, there exists $r \in \mathbb R_{>0}^m$ and power series $F_i \in R\llbracket X_1,\ldots,X_m \rrbracket$, with $i=1,\ldots,n$ such that $B(x,r)=x+\{y \in K^m : |y| < r\} \subseteq U$ and $f(x+y) = (F_1(y),\ldots,F_n(y))$ for all $y \in B(0,r)$.

Furthermore, if $X$ is a topological space, a \emph{chart} on $X$ is a triple $(U,\phi,n)$ where $U$ is an open subset of $X$ and $\phi$ is a homeomorphism of $U$ into $K^{n}$ for a positive integer $n$. We say that $X$ is a \emph{$K$-analytic manifold} if, for each $x \in X$, there exists a chart $(U,\phi,n)$ such that $x \in U$ and, moreover, any two such charts are compatible; if $(V,\psi, m)$ is another chart such that $U \cap V \neq \varnothing$, then the maps $\psi \circ \phi^{-1}|_{\phi(U\cap V)}$ and $\phi \circ \psi^{-1}|_{\psi(U\cap V)}$ are analytic. Such a collection of compatible charts is called an \emph{atlas} on $X$ and its equivalence class is a \emph{$K$-analytic structure} on $X$.

Let $X$ and $Y$ be $K$-analytic manifolds and consider $f \colon X \to Y$. We say that $f$ is \emph{analytic} if $f$ is continuous and there exist atlases of $X$ and $Y$ such that, if $(U,\phi,n)$ and $(V,\psi,m)$ are charts on $X$ and $Y$, respectively, then, setting $W=U \cap f^{-1}(V)$, the function
\[\phi(W) \overset{\phi^{-1}}{\longrightarrow} U \overset{f}{\longrightarrow} V \overset{\psi}{\longrightarrow} \psi(V)\]
is analytic. 

\begin{definition}
    Let $G$ be a topological group with a $K$-analytic structure. We say that $G$ is a \emph{$K$-analytic group} if the functions $(x,y) \mapsto xy$ and $x \mapsto x^{-1}$ are analytic.
\end{definition}

The most ubiquitous examples of such groups are $p$-adic analytic groups; viz.\ when $K=\bbQ_p$ and $R=\bbZ_p$ for a prime $p$. For instance, the groups ${\rm{GL}}_n(\bbQ_p)$ and ${\rm{GL}}_n(\bbZ_p)$ are examples of non-compact and compact $p$-adic analytic groups, respectively. For the convenience of the reader, we also recall the following result which will be useful later in order to deal with analytic functions.

\begin{thm}[{{Weierstra\ss{} Preparation Theorem \cite[Ch. VII \S 3, Prop.~6]{bourbaki1972}}}]\label{thm:weierstrass}
    Let $(R,\mathfrak m)$ be a complete local ring. Suppose $f = \sum_{i\geq 0} a_iX^i \in R\llbracket X \rrbracket$ is non-zero and that there is a positive integer $n \geq 1$ such that $a_i \in \mathfrak m$ for $i < n$ and $a_n \in R\smallsetminus \mathfrak m$. Then there exist an invertible power series $u \in R \llbracket X \rrbracket$ and a polynomial $p (X) = X^n+b_{n-1}X^{n-1}+\ldots+b_0$, with $b_i 
    \in \mathfrak m$, such that $f=pu$.
\end{thm}

\subsection{Representation theory of profinite groups} 

    The main reference for this subsection is \cite[Ch.~5-6]{RZ10}. Let $R$ be a profinite ring. A (left) profinite $R$-module $M$ is a profinite abelian group $M$ endowed with a continuous map $R\times M\to M$ satisfying the usual $R$-module properties. Unless it is specified, every module is meant to be a left profinite module.
	Fixing a profinite commutative ring $R$ and a profinite group $G$, we denote by $\llbracket RG \rrbracket$ the completed group algebra of $G$ over $R$. Recall that, for every positive integer $n$, the group $\mathrm{GL}_n(R)$ with the topology induced by $R$ is profinite. 
	\begin{definition}
		Let $G$ be a profinite group and let $R$ be a profinite commutative ring. An $R$-\emph{representation} of $G$ is a continuous group homomorpshim $\phi\colon G\to \mathrm{GL}_n(R)$ for some $n\in\bbN$.
	\end{definition}
	If $V=R^n$ is the free $R$-module of rank $n$, a representation $\phi\colon G\to \mathrm{GL}_n(R)$ induces an $\llbracket RG \rrbracket$-module structure on $V$ defined by the map $G\times V\to V$ such that $(g,v)\mapsto \phi(g)(v)$. We denote by $V_\phi$ the corresponding $\llbracket RG\rrbracket$-module. Conversely, given an $\llbracket RG \rrbracket$-module $V$, the action of $G$ on $V$ defines a group homomorphism $\phi_V\colon G\to\mathrm{Aut}_R(V)$ that is continuous with respect the compact-open topology on $\mathrm{Aut}_R(V)$. If $V$ is free of rank $n$ as an $R$-module, then $\mathrm{Aut}_R(V)\simeq \mathrm{GL}_n(R)$ as profinite groups, i.e. $\phi_V$ is an $R$-representation of $G$ of dimension $n$. Note that, for every representation $\phi\colon G\to \mathrm{GL}_n(R)$ and every $\llbracket RG \rrbracket$-module $V$ whose underlying structure of $R$-module is free of rank $n$, the functors $V_\phi$ and $\phi_V$ are inverses of each other.
    
	Given representations $\phi\colon G\to \mathrm{GL}_n(R)$ and $\psi\colon G\to \mathrm{GL}_m(R)$, their direct sum is the representation $\phi\oplus\psi\colon G\to \mathrm{GL}_{n+m}(R)$ defined by \[ \phi\oplus\psi(g)=\begin{pmatrix}
		\phi(g) & 0_{n\times m} \\
		0_{m\times n} & \psi(g)
	\end{pmatrix}. \]
	For a subgroup $H\leq G$, we recall the functors $\mathrm{Res}_H^G(-)$ and $\mathrm{Ind}_H^G(-)$. The restriction of an $\llbracket RG \rrbracket$-module $V$ is the $\llbracket RH \rrbracket$-module $\mathrm{Res}_H^G(V)$ defined by restricting the action of $G$ to $H$. The induction of an  $\llbracket RH \rrbracket$-module $W$ is the $\llbracket RG \rrbracket$-module \[ \mathrm{Ind}_H^G(B):= \llbracket RG\rrbracket\hat\otimes_{\llbracket RH \rrbracket}B, \] where $\hat\otimes$ indicates the completed tensor product \cite[Section 5.5]{RZ10}.
    
	Now, suppose that $H\leq_oG$ and consider $\phi\colon H\to \mathrm{GL}_n(R)$ a representation of $H$. Then 
    $\mathrm{Ind}_H^G(\phi)$ is a representation of $G$ on a free $R$-module of rank $n|G:H|$. 
    Moreover, the restriction of $\mathrm{Ind}_H^G(\phi)$ contains a subrepresentation isomorphic to $\phi$. This readily implies the following. 
	\begin{proposition}\label{inducedrepresentationlemma}
		Let $G$ be a profinite group and let $R$ be a profinite commutative ring. Let $H\leq_oG$. If $H$ admits an $R$-representation whose image is not virtually solvable, then so does $G$.
	\end{proposition}

\section{Probabilistic identities in \texorpdfstring{$K$}{K}-analytic groups}

Let $K$ be a non-archimedean local field and $R$ its ring of integers with uniformiser $\pi \in R$. The goal of this section is to prove \cref{thm:analytic}. We begin with a definition.

\begin{definition}
Let $X$ be a topological space with a Borel measure $\mu$. A Borel measurable set $S \subseteq X$ is \emph{locally negligible} at $x \in X$, if there exists an open neighbourhood $U$ of $x$ in $X$ such that $\mu(U \cap S) = 0$. A Borel measurable set $S$ is \emph{locally negligible}, if it is locally negligible at every point.
\end{definition}
\begin{lemma}\label{lem:locallyneg}
Let $X$ be a $\sigma$-compact topological space with a Borel measure $\mu$. 
If $S \subseteq X$ is locally negligible, then $\mu(S) = 0$.
\end{lemma}
\begin{proof}
For every $x \in X$ we pick an open neighbourhood $U_x$ of $x$ with $\mu(U_x \cap S)=0$. Then $X$ is covered by the family $(U_x)_{x \in X}$ of open sets. Since $X$ is $\sigma$-compact, there is a countable subcover $X = \bigcup_{j=1}^\infty U_{x_j}$ and, by countable sub-additivity of $\mu$, we deduce
\[
 \mu(S) \leq \sum_{j=1}^\infty \mu(S \cap U_{x_j} ) = 0. \qedhere
\]
\end{proof}

Given a $K$-analytic group $G$ and a word $w = w(x_1,\ldots,x_r) \in F_r$, the free abstract group of rank $r$, the word map $w_G \colon G^r  \to G$ is analytic, since it comes from successive compositions between the maps obtained from the operation of the group and from taking the inverse of an element, which are themselves analytic. By definition, $w$ is a probabilistic identity for $G$ if, and only if, $\mu(w^{-1}(1_G)) > 0$, so it suffices to study the fibers of analytic maps, which we carry out in the next lemma.

Let $M$ be an $m$-dimensional $K$-analytic manifold. A differential form $\omega$ of degree $m$ gives rise to a measure $\mu_\omega$ on $M$. Let $A \subseteq M$ be a measurable subset and let $(U,\phi,m)$ be a chart. The measure of $\mu_\omega(U\cap A)$ can be calculated in local coordinates as
\[
    \int_{\phi(U\cap A)} |f| dx_1\dots dx_m
\] for an analytic function $f$.
In particular, if $\phi(U\cap A)$ is a zero set for the standard measure on $K^m$, then $\mu_\omega(U\cap A) = 0$. 

We will say that a subset $A \subseteq M$ is
\emph{locally negligible} if it is locally negligible for $\mu_\omega$ for every top-degree form $\omega$.

\begin{lemma}\label{lem:analyticMaps}
Let $M$ and $N$ be $K$-analytic manifolds and let $f \colon M \to N$ be an analytic morphism. Let $y \in N$.
The fibre $S = f^{-1}(y)$ 
is either locally negligible or has a non-empty interior.
\end{lemma}
\begin{proof}
Let $x \in S$, so $f(x) = y$. Assume that $S$ does not contain an open neighbourhood of $x$. We show that this entails local negligibility of $S$ at $x$. The statement is purely local; hence, by picking centered charts at $x$ and $y$ we can assume that $f \colon R^m \to R^n$ is given by convergent power series $(f_1,\dots,f_n)$ in $m$ variables and $x=0$ and $y=0$.

\medskip

\emph{Step 1:} Reduction to $n=1$.\\ 
If $f=(f_1,\dots,f_n)$ is such that $S = f^{-1}(0)$ is not a neighbourhood of  $x=0$, then one of the sets $f_i^{-1}(0)$ is not a neighbourhood of $0$. If $f_i^{-1}(0)$ is locally negligible at $x$, then the smaller set $f^{-1}(0)$ is locally negligible at $0$ as well.

\medskip

\emph{Step 2:} Use \cref{thm:weierstrass} to deduce local negligibility. \\
Let $f(\underline{X})=\sum_{\alpha \in \bbN_0^m} c_\alpha \underline{X}^\alpha$ be given in multi-index notation, i.e., $\underline{X}^\alpha= \prod_{i=1}^m X_i^{\alpha_i}$. Since $f$ converges on $R^m$, the absolute values $|c_\beta|$ tend to zero with $|\beta| \to \infty$.
Consider the graded lexicographic ordering on $\bbN_0^m$.
Since $f(0) = 0$, the constant term is 0, and since $f$ is non-zero, one may find
a minimal $\alpha \in \bbN_0^m$ with $c_\alpha \neq 0$. 
Let $r$ be the total degree of $\alpha$.
The homogeneous summand $f_{r}$ of $f$ of degree $r$ is non-zero, hence there are $\lambda_1,\dots,\lambda_{m-1} \in R$ with $f_{r}(\lambda_1,\dots,\lambda_{m-1},X_m) \neq 0$. 
Substituting $X_i+\lambda_i X_m$ for $X_i$ for all $i < m$, we obtain a non-vanishing monomial of the form $cX_m^{r}$; i.e. we may assume that $\alpha=(0,\dots,0,r)$ from now on.

Our goal will be to apply \cref{thm:weierstrass} to $(R\llbracket X_1,\ldots,X_{m-1} \rrbracket)\llbracket X_m \rrbracket$. Since $|c_\beta|$ tends to $0$, the set $I =\{\beta \in \bbN_0^m : |c_\beta| \geq |c_\alpha| \}$ is finite. Let $\Omega = \{\beta \in I : \beta_i = 0 \text{ for all }1\leq i < m \}$. By assumption $\Omega$ contains $\alpha$ and is non-empty; we may choose a minimal $\gamma_0 \in \Omega$ in the graded lexicographic order such that $|c_{\gamma_0}| \geq |c_\gamma|$ for every $\gamma \in \Omega$. If we substitute $\pi^{k_i}X_i$ for $X_i$ for suitable $k_i$ (i.e., this corresponds to restricting the function to the smaller domain $\pi^{k_1}R \times \cdots \times \pi^{k_{m-1}} R\times R$), we may assume that $|c_\beta| \leq |c_{\gamma_0}|$ for each $\beta \neq \gamma_0$, with the inequality being strict whenever $\beta < \gamma_0$. We multiply with a power of the uniformizer to assume that $|c_{\gamma_0}| = 1$. In either case, it follows that $f$ satisfies the hypothesis of \cref{thm:weierstrass} and we may write $f(X_1,\dots,X_m)= \Bigl(\sum_{i=0}^{r'} t_i(X_1,\dots,X_{m-1})X_m^{i}\Bigr) u(X_1,\dots,X_m)$
where $u \in R[\![X_1,\dots,X_m]\!]$ is invertible and the $t_i$ are power series in $m-1$ variables
with $t_{r'}(X_1,\dots,X_{m-1}) = 1$.
Since $u$ is invertible, its constant term is invertible in $R$; in particular, $u(0)\not=0$ and there exists $k\in{\mathbb N}$ such that $u$ does not have zeros in $(\pi^k R)^m$. For
$x_1,\ldots,x_{m-1}\in\pi^k R$ the polynomial $p(x_1,\dots,x_{m-1},Y) = \sum_{i=0}^r t_i(x_1,\dots,x_{m-1})Y^{i}$ has at most 
$r$ roots.
Now we integrate over the characteristic function of $f^{-1}(0)\cap (\pi^k R)^m$ using Fubini's theorem. Since every finite subset of $R$ is negligible, the inner integral always vanishes and thus $f^{-1}(0) \cap (\pi^k R)^m$ has measure zero with respect to the standard measure on $R^m$.
\end{proof}

\begin{proof}[Proof of \cref{thm:analytic}]
Assume that $w \in F_r$ is a probabilistic identity. Consider the induced word map $w_G \colon G^r \to G$. If $w$ is a probabilistic identity, then $w^{-1}(1_G)$ has positive measure. By \cref{lem:analyticMaps} the fibre $w^{-1}(1_G)$ has a non-empty interior. In particular, we can find an open subgroup $N \trianglelefteq_{\mathrm{o}} G$ such that $g_1N \times g_2N \times \cdots \times g_rN$
is contained in $w^{-1}(1_G)$, so $w$ is an open coset identity.
\end{proof}

We remark that the same result holds if we also consider words with parameters, since similarly the function defined by such a word will be analytic.

\begin{proof}[Proof of \cref{cor:3.5}]
    Suppose $G$ is not virtually solvable. Let $N$ be an open normal $R$-standard subgroup of $G$ and let $Z=Z(N)$. We have that $N/Z$ is linear over $R$ by \cite[Proposition 5.1]{JZ02}, so $G/Z$ is linear as well. Moreover, $G/Z$ is not virtually solvable. Indeed, if $M/Z$ is an open solvable subgroup of $G/Z$, it follows that $M$ is solvable as an extension of the solvable groups $M/Z$ and $Z$. This is a contradiction to the assumption that $G$ is not virtually solvable. Therefore, by \cite[Theorem 1.3]{BG}, there exists a dense free subgroup in $G/Z$, in which case $G/Z$ cannot satisfy a coset identity. Since every probabilistic identity in $G/Z$ is a coset identity by \cref{thm:analytic}, it follows that $G/Z$ is randomly free. In particular, $G$ is randomly free by \cref{cor:quot_rand_free}.
\end{proof}

For completeness we include a proof of the following technical lemma. It is similar to \cite[Lemma~2.1]{BG}, but we need a slightly different statement for which a much shorter proof suffices.%We couldn't locate the precise statement it in the literature.
\begin{lemma}\label{lem:embedding_local}
Let $A$ be a finitely generated integral domain. Then $A$ embeds into the ring of integers $R$ of a local field $K$.
\end{lemma}
\begin{proof}
Let $Q$ be the field of fractions of $A$.
If $Q$ is of characteristic zero, we define $A_0 = \bbQ A$ to be the $\bbQ$-subalgebra of $Q$ generated by $A$. If $Q$ is of characteristic $p$, we write $A_0 = A$. We write $F$ to denote the prime field of $Q$.

We apply Noether normalization \cite[p.~223]{Eisenbud} to $A_0$ to find an $F$-subalgebra $S$ of $A_0$ isomorphic to $F[T_1,\dots,T_n]$ such that $A_0$ is a finitely generated $S$-module; in particular it is an integral extension of $S$.

If $Q$ is of characteristic $0$, we note that $A$ is finitely generated and hence finitely many elements $a_1,\dots,a_k$ generate $A$ as a ring. Now these finitely many elements are integral not only over $\bbQ[T_1,\dots,T_n]$, but already over $\bbZ[1/m][T_1,\dots,T_n]$ for some integer $m \neq 0$.
We define $S =\bbZ[1/m][T_1,\dots,T_n]$ in this case and pick a prime $p$ coprime to $m$.

In the characteristic $0$ case, we put $K_0 = \bbQ_p$ and $R_0 = \bbZ_p$; in the characteristic $p$ case we put $K_0 = \bbF_p(\!(T)\!)$ and $R_0 =\bbF_p\llbracket T \rrbracket$. In both cases $R_0$ contains infinitely many transcendental (algebraically independent) elements over the prime field $F$; cf.~\cite[Lemma 3]{Cassels}. This allows us to embed $S$ into $R_0$; we fix such an embedding $\iota\colon S \to R_0$. It extends to an embedding $\iota\colon F(T_1,\dots,T_n) \to K_0$ and, moreover,
there is a finite algebraic extension $K/K_0$ such that $\iota$ extends further to an embedding $Q \to K$. By construction, $A$ is integral over $S \subseteq R_0$ and hence the image of $A$ lies in the ring of integers $R$ of $K$.
\end{proof}

We can deduce from this Theorem 1.1 and Theorem 1.4 of \cite{LS16}.
\begin{corollary}\label{cor:general_LS_1.4}
    Let $\Gamma$ be a finitely generated linear group. Then $\Gamma$ is either virtually solvable or randomly free. Moreover, every probabilistic identity on $\Gamma$ (possibly with parameters) is a coset identity.
\end{corollary}
\begin{proof}
    Since $\Gamma$ is finitely generated, we can embed
$\Gamma \subseteq \mathrm{GL}_n(A)$ where $A$ is a finitely generated, integral domain. By \cref{lem:embedding_local}, $A$ may be embedded into the ring of integers $R$ of a local field $K$. Hence $\Gamma \subseteq \mathrm{GL}_n(R)$.

Thus the closure $G$ of $\Gamma$ in $\mathrm{GL}_n(R)$ is a compact $K$-analytic group. Moreover, $\Gamma$ is virtually solvable if, and only if, $G$ is virtually solvable. 
Note that $G$ is a continuous quotient of the profinite completion $\widehat{\Gamma}$ of $\Gamma$. If $\Gamma$ is not virtually solvable, then $G$ is not virtually solvable, and hence,  by \cref{cor:3.5}, randomly free.  By \cref{cor:quot_rand_free},
the same applies to $\widehat{\Gamma}$. Moreover, if $\Gamma$ satisfies a probabilistic identity $w$, then so does $G$ and hence $w$ is a coset identity on $G$. Since $\Gamma$ is dense in $G$, we have that $w$ is a coset identity on $\Gamma$ and, hence, on its profinite completion.
\end{proof}
	
\section{Large pro-\texorpdfstring{$p$}{p} groups and applications}\label{sec:free_constr}

To motivate why we restrict our attention to large pro-$p$ groups, we make the following observation. An \emph{upper composition factor} of a profinite group $G$ is a composition factor of $G/N$ for some open normal subgroup $N$ of $G$. If $G$ is a profinite group with infinitely many non-isomorphic non-abelian upper composition factors, then $G$ is randomly free by \cite[Theorem 1.5]{LS18}; this holds for many classes of groups and, in particular, for large profinite groups. In the pro-$p$ case, there is no hope of using this result, given that all the composition factors are isomorphic to $C_p$, so a different method is required to understand whether large pro-$p$ groups are randomly free. 

\begin{proof}[Proof of \cref{vfpprobtits}]
    Let $G$ be a large pro-$p$ group, and let $N\leq_o G$ be an open subgroup that surjects onto a non-abelian free pro-$p$ group $F$ on a profinite space $X$. Since $X$ has finite continuous quotients, we may assume that $X$ is finite, say $|X|=k\geq 2$.
    Let $H$ be a $p$-adic analytic pro-$p$ group that is not solvable and suppose that $d(H)=d$.   By the Schreier formula for free pro-$p$ groups, we may assume that $k\geq d$. Therefore, by the universal property of the free pro-$p$ group,  there is a surjective homomorphism of pro-$p$ groups $\pi\colon F\to H$. Since $H$ is a closed subgroup of $\mathrm{GL}_n(\bbZ_p)$ for some $n$, $\pi$ is a $\bbZ_p$-representation of $F$ whose image is not virtually solvable. By \cref{inducedrepresentationlemma}, for some $m$ there is a $\bbZ_p$-representation $\phi\colon G\to \mathrm{GL}_m(\bbZ_p)$ whose image is a $p$-adic analytic group that is not solvable. Note that we can choose $m=n|G:F|$. Therefore $G$ is randomly free by \cref{cor:3.5}.	
\end{proof}

As a direct application of the above theorem, we show that two classes of pro-$p$ groups satisfy a probabilistic Tits alternative.

\begin{corollary}\label{cor:demushkin}
    Let $G$ be a finitely generated pro-$p$ Demushkin group. Then $G$ is either solvable or randomly free.
\end{corollary}

\begin{proof}
    Let $G$ be a finitely generated Demushkin group; we refer to \cite{Lab67} for the classification of Demushkin groups. If $G$ is finite, then $G \simeq C_2$ is solvable. Suppose $G$ is infinite. If $d(G) \geq 4$, then by examining the possible presentations of $G$, there is a surjection of $G$ onto a non-abelian free pro-$p$ group. 
    
    Suppose $d(G)=2$. By the rank formula, if $U \leq_o G$, then
    \[d(U) = [G:U](d(G) - 2)+2 = 2,\]
    so $G$ is a $p$-adic analytic pro-$p$ group of dimension 2. If $G$ is abelian, it is solvable. If $G$ is not abelian, then $G^{\rm{ab}} \simeq \bbZ/q\bbZ \times \bbZ_p$, where $q$ is a power of $p$. This implies that $[G,G]$ is a $p$-adic analytic group of dimension 1, thus solvable. 
    
    Now, suppose $d(G)=3$. By the rank formula and the dichotomy of subgroups of Demushkin groups, there exists $H\leq_o G$ which is Demushkin and such that $d(H)>3$, so $H$ surjects onto a non-abelian free pro-$p$ group and consequently $G$ is large. Thus, the result follows from \cref{vfpprobtits}.
\end{proof}

Notice that the cases in which $d(G) \neq 3$ do not require much machinery, since they are either solvable or directly surject onto a non-abelian free pro-$p$ group. However, when $d(G)=3$, there are \emph{no} surjections from $G$ onto a non-abelian free pro-$p$ group by \cite{S74}, so \cref{vfpprobtits} is essential.

\begin{corollary}\label{freeprod}
    Let $G$ be a pro-$p$ group that splits as a non-trivial free pro-$p$ product of pro-$p$ groups. If $p$ is odd, then $G$ is randomly free. If $p=2$ and $G\not\simeq C_2\amalg C_2$, then $G$ is randomly free. In particular, $G$ is either solvable or randomly free.
\end{corollary}
\begin{proof}
    If $p$ is odd, then $G$ projects onto $C_p \amalg C_p$. The kernel of the map $C_p \amalg C_p \to C_p \times C_p$ is non-abelian free pro-$p$ by the Kurosh subgroup theorem (see e.g. \cite[Theorem 7.3.1]{R17}) and by \cite[Corollary 7.1.8]{R17}. If $p=2$, consider $G = \amalg_X G_x$, where $X$ is a profinite space. If $X=\{x_1,x_2\}$, then either $G_{x_1} = G_{x_2} \simeq C_2$ and $G$ is solvable or $G_{x_1}$, without loss of generality, has a finite quotient $Q$ of order greater than 4, in which case $G$ surjects onto $Q \amalg C_2$. If $|X| > 2$, then $G$ surjects onto $C_2 \amalg C_2 \amalg C_2$, which is virtually non-abelian free pro-2 by the same argument as in the first case.
\end{proof}

Before continuing, we recall the following standard fact. If $G$ is a pro-$p$ group and $H \leq G$ is a proper subgroup, then there exists $N\unlhd_o G$ such that $HN$ is a proper open subgroup of $G$. Since $G/N$ is nilpotent, the normal closure of $HN/N$ in $G/N$ is a proper normal subgroup of $G/N$, whose preimage in $G$ gives us a proper open normal subgroup of $G$ that contains $H$. More generally, this shows that the normal closure $H^G$ of $H$ in $G$ is a proper normal subgroup of $G$.

For the remaining proofs in this section, we use profinite Bass-Serre theory; we refer the reader to \cite{R17} for the basic definitions and for a deeper exposition on the topic. A \emph{finite directed graph} is a quintuple $(\Gamma, V(\Gamma), E(\Gamma), d_0,d_1)$ where: $\Gamma$ is a finite set which is a disjoint union of its subsets $V(\Gamma)$ and $E(\Gamma)$--called the \emph{vertex set} and \emph{edge set} of $\Gamma$, respectively--and functions $d_i \colon \Gamma \to V(\Gamma)$, for $i \in \{0,1\}$, such that their restriction to $V(\Gamma)$ is the identity--$d_0(e)$ and $d_1(e)$ are called the \emph{initial} and \emph{terminal} vertices of an edge $e \in E(\Gamma)$, respectively.
%All graphs $\Gamma$ that we consider are directed, so we have functions $d_i \colon \Gamma \to V(\Gamma)$, for $i \in \{0,1\}$, where $d_0(m)$ and $d_1(m)$ are the \emph{initial} and \emph{terminal} vertices of $m$,
Let $(\mathcal G, \Gamma)$ be a graph of pro-$p$ groups over a finite connected graph $\Gamma$, so to each $m \in \Gamma$ we associate a pro-$p$ group $\mathcal G(m)$. If $m$ is a vertex (resp. edge), we say that $\mathcal G(m)$ is a vertex group (resp. edge group). Furthermore, if $e$ is an edge and $d_0(e)$ and $d_1(e)$ are its initial and terminal vertices, respectively, then we have monomorphisms $\partial_0 \colon\mathcal G(e) \to \mathcal G(d_0(e))$ and $\partial_1 \colon \mathcal G(e) \to \mathcal G(d_1(e))$. Let $T$ be a maximal subtree of $\Gamma$. The \emph{fundamental pro-$p$ group of $(\mathcal G,\Gamma)$} will be denoted by $\Pi=\Pi_1(\mathcal G, \Gamma)$ and is defined as
\[\left( \coprod_{v \in V(\Gamma)} \mathcal G(v) \right) \amalg F,\]
where $F$ is the free pro-$p$ group on the set $\{t_e \mid e \in E(\Gamma)\}$, modulo the normal closure $N$ of the set
\[\{t_e \mid e \in E(T)\} \cup \{\partial_0(x)^{-1}  t_e\partial_1(x)t_e^{-1} \mid x \in \mathcal G(e), e\in E(\Gamma)\}.\]
We remark that the resulting group is independent of the choice of maximal subtree. If all the vertex and edge groups embed in $\Pi$, then we say that $(\mathcal G, \Gamma)$ is \emph{injective.} That is not always the case, but we can replace $(\mathcal G,\Gamma)$ by an injective graph of groups $(\tilde{\mathcal G},\Gamma)$ such that $\Pi_1(\mathcal G, \Gamma)=\Pi_1(\tilde{\mathcal G},\Gamma)$; notice that the underlying graph does not change (essentially, we replace the vertex and edge groups by their canonical images in the fundamental group). Moreover, we say that an injective graph of groups is \emph{reduced} if the monomorphisms of edge groups to vertex groups are not surjective for all edges that are no loops. When $\Gamma$ is finite, we can also replace a non-reduced graph of groups by a reduced one with the same fundamental group, which is achieved by collapsing the `superfluous' edges. Details for these procedures may be found in \cite[Sections 6.4, 7.2]{R17}. We recall also that to a graph of pro-$p$ groups $(\mathcal G,\Gamma)$ over a finite graph $\Gamma$ we can associate a standard graph $S=S(\mathcal G,\Gamma)$ (which is a $p$-tree) on which $G=\Pi_1(\mathcal G,\Gamma)$ acts and such that $G\setminus S$, the set of orbits of the action, is equal to $\Gamma$ (\cite[Section 6.3]{R17}).

For a vertex $v \in V(\Gamma)$, we denote by $X_v$ the set of edges in $\Gamma$ incident to $v$ and we say that $v$ is a \emph{pending vertex} if $|X_v|=1$. If $e \in X_v$, let $J(v,e) \subseteq \{0,1\}$ be the set of indices $i$ such that $d_{i}(e)=v$.

\begin{lemma}\label{cor:graphs}
    Let $(\mathcal G, \Gamma)$ be an injective reduced graph of pro-$p$ groups over a finite connected graph $\Gamma$ and let $T$ be a maximal subtree of $\Gamma$. Define
    \[\tilde{\mathcal G}(v):=\mathcal G(v)/\langle \partial_i(\mathcal G(e)) \mid e \in X_v,\hspace{0.2em} i \in J(v,e)\rangle^{\mathcal G(v)}.\]
    If $\Gamma$ is not a tree, suppose that there is a vertex $v$ such that $\tilde{\mathcal G}(v)$ is non-trivial or that $|E(\Gamma) \smallsetminus E(T)| \geq 2$. If $\Gamma$ is a tree, suppose that $|V(\Gamma)|\geq 2$ and, if $p=2$, that $\tilde{\mathcal G}(v)$ is non-trivial for at least 3 vertices or that $|\tilde{\mathcal G}(v)| \geq 4$ for one vertex.
    %either $\Gamma$ has at least 3 pending vertices or that one of the pending vertices $v=d_i(e)$ of $\Gamma$ is such that $|\mathcal G(v):\partial_i(\mathcal G(e))^{\mathcal G (v)}| \geq 4$
    Then $\Pi_1 (\mathcal G, \Gamma)$ is randomly free.
\end{lemma}
\begin{proof}
    Define a graph of pro-$p$ groups $(\tilde{\mathcal G},\Gamma)$ by setting $\tilde{\mathcal G}(v)$ as the vertex groups and letting the edge groups be trivial. Fix $T$ a maximal subtree of $\Gamma$.
    % Let $T$ be a maximal subtree of $\Gamma$. For $v \in V(\Gamma)$, let
    % \[\tilde{\mathcal G}(v) := \mathcal G(v)/\langle \partial_i(\mathcal G(e)) \mid e \in X_v,\hspace{0.2em} i \in J(v,e)\rangle^{\mathcal G(v)}\]
    % and, for $e \in E(\Gamma)$, let $\tilde{\mathcal {G}}(e) = 1$. 
    Thus, $(\tilde{\mathcal G},\Gamma)$ is a graph of pro-$p$ groups and
    \[\tilde{\Pi} = \Pi_1(\tilde{\mathcal G},\Gamma) = \left( \coprod_{v \in V(\Gamma)} \tilde{\mathcal G}(v) \right) \amalg L,\]
    where $L$ is a free pro-$p$ group with free generating set $\{t_e \mid e \in E(\Gamma) \smallsetminus E(T)\}$. With respect to $T$, we have that $\Pi = \Pi_1(\mathcal G, \Gamma)$ is the group
    \[\left( \coprod_{v \in V(\Gamma)} \mathcal G(v) \right) \amalg F,\]
    where $F$ is the free pro-$p$ group on the set $\{t_e \mid e \in E(\Gamma)\}$, modulo the normal closure $N$ of the set
    \[\{t_e \mid e \in E(T)\} \cup \{\partial_0(x)^{-1}t_e\partial_1(x)t_e^{-1} \mid x \in \mathcal G(e), e\in E(\Gamma)\}.\]
    By the universal property of the free pro-$p$ product, we have a surjective homomorphism
    \[\left( \coprod_{v \in V(\Gamma)} \mathcal G(v) \right) \amalg F \to \tilde{\Pi}.\]
    By construction of $(\tilde{\mathcal G}, \Gamma)$, this map factors through $N$ and we obtain a surjection $\Pi \to \tilde{\Pi}$.

    We divide the proof in two cases. Suppose $\Gamma$ is not a tree. If every $\tilde{\mathcal G}(v)$ is trivial and $|E(\Gamma) \smallsetminus E(T)| \geq 2$, then $\tilde{\Pi}$ is a non-abelian free pro-$p$ group and thus randomly free. If there exists a vertex $v$ such that $\tilde{\mathcal G}(v)$ is non-trivial, then $\tilde{\Pi}$ is a non-trivial free pro-$p$ product and thus randomly free by \cref{freeprod}.

    Suppose $\Gamma$ is a tree. For $p$ odd, since $|V(\Gamma)| \geq 2$, there at least two pending vertices in $\Gamma$, say $v_1$ and $v_2$. The graph of pro-$p$ groups $(\tilde{\mathcal G},\Gamma)$ is injective and reduced by assumption, so $\tilde{\mathcal G}(v_1)$ and $\tilde{\mathcal G}(v_2)$ are non-trivial. Again, $\tilde{\Pi}$ is a non-trivial free pro-$p$ product.
    The case $p=2$ may be proved by a similar argument.
\end{proof}

The following is not complicated, but we include a proof for completeness.

\begin{lemma}\label{lem:circuit}
    Let $\Gamma$ be a connected graph such that $|V(\Gamma)|=|E(\Gamma)| = n \geq 2$ and suppose that $\Gamma$ has no pending vertices. Then $\Gamma$ is a circuit.
\end{lemma}
\begin{proof}
    We prove it by induction on $n$. The case $n=2$ is clear. Suppose now that $n>2$. It is a standard fact that $\Gamma$ has exactly one cycle. If we remove an edge $e$ in this cycle, then $T=\Gamma \smallsetminus \{e\}$ is a tree. In that case, $T$ has at least two pending vertices; let $v_1$ and $v_2$ be two of them. We conclude that $d_{1-i}(e) = v_1$ and $d_i(e) = v_2$ for some $i \in \{0,1\}$, otherwise it would contradict the fact that $\Gamma$ has no pending vertices. It follows that the valency of both $v_1$ and $v_2$ are equal to 2. If we collapse the edge $e$ to $v_2$, we obtain a circuit by induction. Upon reinserting $e$, we conclude that $\Gamma$ is a circuit.
\end{proof}

Combining \cref{cor:graphs} and \cref{lem:circuit}, one sees that the only graphs of pro-$p$ groups $(\mathcal G,\Gamma)$ left to tackle are when $p=2$ and $\Gamma$ is a line or when $p$ is any prime and $\Gamma$ is a circuit, which is what we do next.

We will say that a reduced injective graph of pro-$p$ groups is of \emph{type I} if it is of the form
\begin{center}
    \begin{tikzpicture}
        \filldraw[black] (0,0) circle (2pt);
        \draw[->,black] (0.5,0) circle [radius=15pt];
        \node[left=2pt] at (0,0) {$H$};
        \node[right=2pt] at (1,0) {$K$};
    \end{tikzpicture}
\end{center}
and $\partial_i(K) = H$ for both $i \in \{0,1\}$, and of \emph{type II} if it is of the form 
\[
\begin{tikzpicture}
    \filldraw[black] (-1,0) circle (2pt) 
                     (1,0) circle (2pt);
    \draw (-1,0) -- (1,0) node [midway, above=2pt, fill=white] {$H$};
    \node[above=3pt] at (-1,0) {$G_1$};
    \node[above=3pt] at (1,0) {$G_2$};
\end{tikzpicture}
\]
where $H$ has index 2 in both $G_1$ and $G_2$. For example, every finite central extension $G_c=C_{2c}\amalg_{C_c}C_{2c}$ of the infinite dihedral pro-$2$ group $C_2\amalg C_2$ is of type II. Note that $G_c$ is solvable.

Notice that we do not have to consider HNN-extensions such that $\partial_0(K) \neq H = \partial_1(K)$, because, in this case, the HNN-extension is not proper. Indeed, let $G=\mathrm{HNN}(H,K,\partial)$ be a pro-$p$ HNN-extension such that $K$ is a proper subgroup of $H$ and $\partial \colon K \to H$ is an isomorphism. Suppose that $G$ is proper. Since $K$ is a proper subgroup of $H$, then there exists $U \unlhd_o H$ such that $KU \neq H$. By \cite[Theorem 1.3]{Zoe94}, there exists $V \unlhd_o H$ such that, in particular, $V \leq U$ and $\partial$ induces an isomorphism $\tilde \partial \colon KV/V \to H/V$. The latter statement contradicts the fact that $KV \neq H$, so $G$ cannot be proper. In other words, ascending pro-$p$ HNN-extensions are never proper.

Recall that if $(\mathcal G, \Gamma)$ is a graph of pro-$p$ groups over a finite graph $\Gamma$ and $T$ is a spanning tree of $\Gamma$, then $E(\Gamma)-E(T) \geq 2$ implies that $\Pi_1(\mathcal G, \Gamma)$ surjects onto a non-abelian free pro-$p$ group, as in the proof of \cref{cor:graphs}.

\begin{thm}\label{thm:graphs_of_groups}
     Let $(\mathcal G, \Gamma)$ be a reduced injective graph of pro-$p$ groups and let $G=\Pi_1(\mathcal G, \Gamma)$. If $p$ is odd and $(\mathcal G, \Gamma)$ is not of type I, then $G$ is randomly free. If $p=2$ and $(\mathcal G, \Gamma)$ is neither of type I nor II, then it is randomly free.
\end{thm}
\begin{proof}
    By \cref{lem:circuit} and \cref{cor:graphs}, it suffices to consider graphs of groups whose underlying graph $\Gamma$ is a circuit that is not of type I or a line that is not of type II. Consider $S = S(\mathcal G, \Gamma)$ the standard graph of $(\mathcal G, \Gamma)$, let $N$ be an open normal subgroup of $G$ and denote the vertex and edge groups of $(\mathcal G, \Gamma)$ by $G_v$ and $G_e$, respectively. Since $G \setminus S$ is finite, then $N \setminus S = \Gamma_N$ is also finite, so, by \cite[Theorem 6.6.1]{R17}, $N$ splits as the fundamental pro-$p$ group $\Pi_1(\mathcal G_N,\Gamma_N)$ of a graph of pro-$p$ groups over the finite graph $\Gamma_N$. In this case, we have that
    \[|V(\Gamma_N)| = \sum_{v \in V(\Gamma)} |G/G_vN|\]
    and
    \[|E(\Gamma_N)| = \sum_{e \in E(\Gamma)} |G/G_eN|.\]
    If $T_N$ is a spanning tree for $\Gamma_N$, then we further have that $|E(T_N)|=|V(\Gamma_N)|-1$. From
    \[|E(\Gamma_N)|-|E(T_N)| = |E(\Gamma_N)|-|V(\Gamma_N)|+1\]
    and \cref{vfpprobtits} it follows that, in order to conclude that $G$ is randomly free, it suffices to find $N\unlhd_o G$ such that $|E(\Gamma_N)|-|V(\Gamma_N)|> 0$. First, consider the case where $G=\mathrm{HNN}(G_v,G_e,f)$. Since $(\mathcal G, \Gamma)$ is not of type I, then the embeddings of the edge group in $G_v$ are proper. So we can choose $N$ such that $|G_vN/G_eN| > 1$. Thus, we have
    \[|E(\Gamma_N)|-|V(\Gamma_N)| = |G/G_eN|- |G/G_vN| = |G/G_vN|(|G_vN/G_eN|-1),\]
    which is greater than 0. Second, for circuits with at least two vertices or lines not of type II, we have
    \begin{align*}
        |E(\Gamma_N)|-|V(\Gamma_N)| & = \sum_e |G/G_eN|- \sum_v |G/G_vN| \\ & = \sum_{v \in V(\Gamma)}|G/G_vN|\left(\sum_{e\in X_v}\cfrac{1}{2}|G_vN/G_eN|-1 \right),
    \end{align*}
    where the half factor appears from the fact that each edge is incident to two distinct vertices. Since there are finitely many edge groups and the graph of groups is reduced, one can choose $N$ such that $|G_vN/G_eN| > 1$ for every $e \in E(\Gamma)$, so every summand on the right is non-negative. If $p=2$ and $(\mathcal G, \Gamma)$ is not of type I nor II, then either we have an edge group of index greater than 2 in one of its incident vertex groups or, if every edge group has index 2 in its incident vertex groups, then there will be a vertex $v$ such that $|X_v| \geq 2$. In any case, the quantity above is greater than 0 and thus $N$ is large. If $p$ is odd, the same conclusion directly holds. Thus, $G$ is also large and thus randomly free by \cref{vfpprobtits}.
\end{proof}

We recall the following definition of pro-$p$ analogues of limit groups which first appeared in \cite{KZ}. Let $p$ be a prime number and let $\mathcal G_0$ be the class of free pro-$p$ groups of finite rank. Define $\mathcal G_n$ recursively as follows: a group $G_n$ is in $\mathcal G_n$ if $G_n = G_{n-1} \amalg_C A$, where $G_{n-1} \in \mathcal G_{n-1}$, $A$ is free abelian pro-$p$ of finite rank and $C$ is a self-centralizing procyclic subgroup of $G_{n-1}$ and also a proper direct summand of $A$. A pro-$p$ group $H$ is in the class $\mathcal L$ if it is a finitely generated pro-$p$ subgroup of a group in $\mathcal G_n$ for some $n\geq 0$. The \emph{weight} of $H$ is the smallest such $n$. Our aim will be to prove that groups in $\mathcal L$ are either solvable or randomly free. In light of \cite[Theorem B]{SZ14}, the class $\mathcal L$ is suitable for a combinatorial approach.

\begin{lemma}\label{lem:index_normal_closure}
    Let $H \in \mathcal L$ be such that $d(H) > 1$ and let $x \in H$. Then $|H:\langle x \rangle^H| \geq 3$.
\end{lemma}
\begin{proof}
    Since $H$ is not procyclic, then $\langle x \rangle^H$ is necessarily a proper normal subgroup of $H$, as remarked earlier, so we can restrict ourselves to the case $p=2$. If $d(H)=2$, then it is either the free pro-2 group of rank 2 or $\mathbb Z_2^2$ (\cite[Theorem 7.3]{KZ}) and for both the conclusion holds. Suppose $d(H)>2$, let $N=\langle x \rangle^H$ and let $\Phi(H)$ be the Frattini subgroup of $H$. If $x \in \Phi(H)$, the statement follows since $d(H)>2$. Suppose now that $x \notin \Phi(H)$ and let $H'=[H,H]$. Notice that $d(H/H') > 2$, so $\langle x \rangle^HH'/H' = \langle x \rangle H'/H'$ is a procyclic subgroup of a group which is at best 3-generated. Thus, we have $|H:\langle x \rangle^H| \geq 4$.
\end{proof}

\begin{corollary}
    The groups in $\mathcal L$ are either solvable or randomly free.
\end{corollary}
\begin{proof}
    Let $G$ be a pro-$p$ group in $\mathcal L$. If $G$ has weight 0, then it is a subgroup of a finitely generated free pro-$p$ group, so it's either solvable or randomly free. Suppose that $G$ has weight $n\geq 1$. By \cite[Theorem B]{SZ14}, $G$ splits as the fundamental group of a graph of pro-$p$ groups $(\mathcal G, \Gamma)$ over a finite graph $\Gamma$ such that the vertex groups are finitely generated and the edge groups are either procyclic or trivial. Since $\Gamma$ is finite, we may suppose the graph of groups is injective and reduced; this means that each vertex and edge group is also in $\mathcal L$. Suppose $G$ is of type I; in this case, since the edge group is procyclic, so is the vertex group, thus $G$ is 2-generated and, by \cite[Theorem 7.3]{SZ14}, isomorphic either to $\mathbb Z_p^2$ or to $F_2$, the free pro-$p$ group in 2 generators. Suppose $G$ is of type II. Since the edge group is procyclic, a combination of \cref{lem:index_normal_closure}, \cref{cor:graphs} and the fact that the groups in $\mathcal L$ are commutative-transitive (\cite[Theorem 5.1]{KZ}) yields that $G$ is abelian. Finally, if $(\mathcal G, \Gamma)$ is not of type I or II, then $G$ is randomly free by \cref{thm:graphs_of_groups}. 
\end{proof}

It remains to understand what happens with the other two essential types of free constructions: free products with amalgamation and HNN-extensions. For these, one cannot prove a general probabilistic Tits alternative result. Indeed, take $R_p$ to be the finitely generated pro-$p$ group which satisfies an identity, but is not solvable, constructed in \cite[Corollary 10]{CM07}. The group 
\[(C_2 \times R_2) \amalg_{R_2}(C_2 \times R_2)\] is a proper free pro-2 product with amalgamation and, if $\varphi \in \mathrm{Aut}(R_p)$ is a $p$-element, the group 
\[\mathrm{HNN}(R_p,R_p,\varphi) = R_p \rtimes_\varphi \mathbb Z_p \] 
is a proper pro-$p$ HNN-extension. They are not solvable, but they satisfy an identity, so they are not randomly free. Notice that both of these examples fall into the type I and type II cases defined previously. Outside of that, we are able to give a full answer.

\begin{corollary}
    Let $G$ be a pro-$p$ group that splits either as proper pro-$p$ HNN-extension $\mathrm{HNN}(H,A,f)$ with $A\neq H\neq f(A)$ or as a proper free pro-$p$ product with amalgamation $G_1\amalg_H G_2$ with $\max \{|G_1:H|,|G_2:H|\} > 2$. Then $G$ is randomly free.
\end{corollary}

\begin{proof}
    This follows from the fact that, in either case, $G$ splits a fundamental group of a graph of pro-$p$ groups over a finite graph which is neither of type I nor of type II, so we are within the hypothesis of \cref{thm:graphs_of_groups}.
\end{proof}

\begin{remark}
    In this remark, we distinguish between topological and abstract generation. For a profinite group $G$ and a positive integer $k$, consider the measurable sets
    \[Y_k = \{(g_1,\ldots,g_k) \in G^k : \langle g_1,\ldots,g_k \rangle \text{ is a free abstract group of rank $k$}\}\]
    and
    \[X_k = \{(g_1,\ldots,g_k) \in G^k : \overline{\langle g_1,\ldots,g_k \rangle} = G \}.\]
    If $G$ is a finitely generated pro-$p$ group, then $G$ is \emph{positively finitely generated}, that is, there exists a positive integer $k$ such that $\mu(X_k)>0$. In fact, if $d(G) = d$, then we can take $k=d$ by recalling that in this case the Frattini subgroup $\Phi(G)$ is open in $G$ and elements $g_1,\ldots,g_d \in G$ topologically generate $G$ if, and only if, their reductions modulo $\Phi(G)$ generate $G/\Phi(G)$. It follows that if $G = \overline{\langle g_1,\ldots,g_d\rangle}$, then $g_1 \Phi(G) \times \ldots \times g_d \Phi(G) \subseteq X_d$.
    
    Suppose also that $G$ is randomly free. This implies that $\mu(Y_k) = 1$ for every positive integer $k$ and, moreover, that $\mu(Y_d \cap X_d) > 0$. So, for all the examples of finitely generated randomly free pro-$p$ groups considered above, we can not only find dense free abstract subgroups, but we can do so with positive probability.
\end{remark}

   Recall that the \emph{rank gradient} of a finitely generated pro-$p$ group $G$ is defined as
	\begin{equation*}
		\mathrm{RG}(G)=\underset{H\leq_oG}{\mathrm{inf}}\frac{d(H)-1}{|G:H|},
	\end{equation*}
	where $d(H)$ denotes the minimal number of generator of the pro-$p$ group $H$. Moreover, if $M\leq_oG$, then using the Schreier inequality it can be shown that $\mathrm{RG}(M)=|G:M|\mathrm{RG}(G)$. In particular, pro-$p$ groups that have an open subgroup of positive rank gradient are of positive rank gradient. Note that many groups that we have shown to be randomly free have positive rank gradient, see for example \cite[Section 1]{GGK19}.
    
    On the other hand, by \cite[Theorem 1.7]{Ni}, a finitely generated pro-$p$ group of positive rank gradient contains a free non-abelian dense subgroup, so it cannot satisfy any coset identity. We do not know if this result can be generalised to probabilistic identities, so we leave it as a problem.
	\begin{prob}\label{prob:rank_grad}
		Let $G$ be a finitely generated pro-$p$ group of positive rank gradient. Is $G$ randomly free? 
	\end{prob}

	\section{Torsion probabilistic identities}\label{sec:torsion}
    \subsection{Conjugacy classes and virtually free pro-\texorpdfstring{$p$}{p} groups}\label{subsec:virt_free_torsion}
	In this section, we study how to measure conjugacy classes in countably based profinite groups. As an application, using a Lie-theoretic argument that we think is of its own interest, we will give an alternative proof of \cref{vfpprobtits} for torsion words $w=x^k$, $k\in\bbN$.
    
	Given $x\in G$, we denote by $x^G:=\{x^g\mid g\in G\}$ its conjugacy class in $G$. It is clear that $x^G$ is a closed subset of $G$.
	\begin{proposition}\label{pr:ccmeasure}
		Let $G$ be a countably based profinite group and let $x\in G$. Given $\{N_i\}_{i\geq 1}$ a descending chain of open normal subgroups of $G$ with trivial intersection, define 
		\begin{align*}
			& c_1(x)=\underset{i\geq 1}{\mathrm{sup}}|C_{G/N_i}(xN_i)|;\\
			& c_2(x)=\underset{N\unlhd_oG}{\mathrm{sup}}|C_{G/N}(xN)|;\\
			& c_3(x)=\underset{K\unlhd G}{\mathrm{sup}}|C_{G/K}(xK)|,
		\end{align*}
		where $c_3(x)=\infty$ if $C_{G/K}(xK)$ is infinite for some $K\unlhd G$. Then $c_1(x)=c_2(x)=c_3(x):=c(x)$ and \begin{equation}\label{eq:ccmeasure}
			\mu(x^G)=\frac{1}{c(x)}.
		\end{equation}
	\end{proposition}
	\begin{proof}
		Since for any $N\unlhd_o G$ one has $x^GN/N=(xN)^{G/N}$, it follows that
		\begin{equation*}\label{eq:orbstab}
			|x^GN/N|=|(xN)^{G/N}|=|G/N:C_{G/N}(xN)|=\frac{|G/N|}{|C_{G/N}(xN)|}. 
		\end{equation*} 
		By \cref{quotient}, one obtains that 
        
		\begin{equation*}
			\mu(x^G)=\underset{i\geq 1}{\mathrm{inf}}\frac{1}{|C_{G/N_i}(xN_i)|}=\underset{N\unlhd_oG}{\mathrm{inf}}\frac{1}{|C_{G/N}(xN)|}. 
		\end{equation*} Therefore $c_1(x)=c_2(x)$ and $\mu(x^G)=1/c_i(x)$ for $i=1,2$. Since $c_2(x)\leq c_3(x)$, it remains to verify that $c_3(x)\leq c_2(x)$. Let $K\unlhd G$ and consider the quotient group $G/K\simeq \underset{N\unlhd_oG}\varprojlim\frac{G}{KN}$.
        Since $C_{G/K}(xK)$ is a closed subgroup of $G/K$, one has that 
		\begin{equation}\label{eq:centrquotient}
			C_{G/K}(xK)\simeq\underset{N\unlhd_oG}\varprojlim\frac{C_{G/K}(xK)KN}{KN}.
		\end{equation}
		Moreover, $C_{G/K}(xK)KN/KN\leq C_{G/KN}(xKN)$ for every $N\unlhd_oG$. Therefore \cref{eq:centrquotient} implies the following chain of inequalities 
		\begin{equation*}
			|C_{G/K}(xK)|=\underset{N\unlhd_oG}{\mathrm{sup}}\left\vert\frac{C_{G/K}(xK)KN}{KN}\right\vert\leq \underset{N\unlhd_oG}{\mathrm{sup}}|C_{G/KN}(xKN)|\leq c_2(x).
		\end{equation*}
		Since the previous inequality holds for every $K\unlhd G$, it follows that $c_3(x)\leq c_2(x)$.
	\end{proof}
	An immediate consequence of \cref{pr:ccmeasure} is the following sufficient condition to have conjugacy classes of measure zero in a countably based profinite group.
	\begin{corollary}\label{criterium}
		Let $G$ be a countably based profinite group and let $x\in G$. Suppose that there exists $K\unlhd G$ such that $|C_{G/K}(xK)|=\infty$. Then $\mu(x^G)=0$.
	\end{corollary}

	Next we introduce the Lie rings in which we are interested. Given an abstract or profinite group $G$, let \[ L(G)=\bigoplus_{i\geq 1}\frac{\gamma_i(G)}{\gamma_{i+1}(G)}\] be the associated Lie ring with respect to its lower central series $\gamma_i(G)$. Recall that for a given  Lie ring $L$ a Lie ring endomorphism $\alpha\colon L\to L$ is called regular if $x^\alpha\neq x$ for every $x\in L$, $x\neq 0$. Moreover, 
    by a result of V.A.~Kreknin,
    a Lie ring that admits a finite order automorphism that is regular needs to be solvable (\cite[Theorem 4.3.1]{Khu93}).
	\begin{lemma}\label{lm:abstract}
		Let $F$ be a finitely generated non-abelian free abstract group and let $\alpha\in\mathrm{Aut}(F)$ be 
        of finite order. Then there is $n\in\bbN$ such that $\alpha$ has infinitely many fixed points on $F/\gamma_{n+1}(F)$.
	\end{lemma}
	\begin{proof}
		Suppose that $F$ is the free group on $X$, where $|X|\geq2$. $L(F)$ is not solvable since by \cite[Theorem 4.6.1]{Ser65} it is the free Lie ring on $X$. Then, being of finite order, $\alpha$ cannot be regular on $L(F)$. That is, there is a non-trivial element $x\gamma_{n+1}(F)\in\gamma_{n}(F)/\gamma_{n+1}(F)$ such that $x^\alpha\gamma_{n+1}(F)=x\gamma_{n+1}(F)$. Since $F/\gamma_{n+1}(F)$ is torsion-free and every element of $\{x^i\gamma_{n+1}(F)\}_{i\geq1}$ is fixed by $\alpha$, 
        this yields the claim.
      \end{proof}
	 We show that the same statement holds in the pro-$p$ setting. 
	\begin{lemma}\label{pr:adaptation}
		Let $F$ be a finitely generated non-abelian free pro-$p$ group and let $\alpha\in\mathrm{Aut}(F)$ be of finite order. Then there is $n\in\bbN$ such that $\alpha$ has infinitely many fixed points on $F/\gamma_{n+1}(F)$.
	\end{lemma}
	\begin{proof} 
		Let $E$ be an abstract non-abelian free group whose pro-$p$ completion is $F$. If we denote by $E_i$ the $i$-th term of the lower central series of $E$ and by $F_i$ the analogous for $F$, we have that $F_i\cap E=E_i$ by \cite[Lemma 2.6]{Lub82}. This implies that the Lie ring homomorphism $\phi\colon L(E)\to L(F)$ defined componentwise as
		\begin{align*}
			\phi\colon &E_i/E_{i+1}\to F_i/F_{i+1}\\
			& xE_{i+1}\to xF_{i+1}
		\end{align*}
	    is injective. Indeed, if $x,y\in E_i$ such that $xF_{i+1}=yF_{i+1}$, then $xy^{-1}\in F_{i+1}\cap E=E_{i+1}$, that is $xE_{i+1}=yE_{i+1}$. It follows that $L(F)$ is not solvable since it contains a sub-Lie ring that is not solvable. Since $F/F_i$ is torsion-free for every $i$, the argument used in the proof of \cref{lm:abstract} can be transferred verbatim.
	\end{proof}
	It is now possible to measure conjugacy classes of torsion elements in a particular class of profinite groups.
	\begin{proposition}\label{generalized}
		Let $G$ be a countably based profinite group and let $F\unlhd G$ be a finitely generated non-abelian free pro-$p$ group. Then $\mu(t^G)=0$ for every torsion element $t\in G$.
	\end{proposition}
	\begin{proof}
		Let $c_t\in\mathrm{Aut}(G)$ be the conjugation by $t$. Since $F\unlhd G$, we have that ${c_t}_{|_F}\in\mathrm{Aut}(F)$ and $|c_t|=|t|<\infty$. By \cref{pr:adaptation}, $c_t$ has infinitely many fixed points on $F/\gamma_{n+1}(F)$ for some $n$. Recall that $\gamma_{n+1}(F)$ is a characteristic subgroup of $F$ that is normal in $G$, therefore it is a normal subgroup of $G$. Then  $C_{G/\gamma_{n+1}(F)}(t\gamma_{n+1}(F))$ is infinite since it contains $C_{F/\gamma_{n+1}(F)}(t\gamma_{n+1}(F))$. Therefore $\mu(t^G)=0$ by \cref{criterium}.
	\end{proof}
	In particular, if $G$ is a finitely generated virtually free pro-$p$ group that is not virtually procyclic, then $\mu(t^G)=0$ for every torsion element $t\in G$. Moreover, there are just finitely many conjugacy classes of torsion elements by \cite[Lemma 8]{HZ08}. Therefore one obtains an alternative proof of \cref{vfpprobtits} for torsion words $w=x^k$, $k\in\bbN$.

\subsection{Torsion in compact \texorpdfstring{$p$}{p}-adic analytic groups}\label{subsec:p-adic_torsion}

We now focus on torsion identities in compact $p$-adic analytic groups. We start by introducing some notions from \cite{Ser65} that will be used throughout this section. Given a $K$-analytic group $G$, we denote by $\mathcal{L}(G)$ its associated $K$-Lie algebra, defined as the tangent space at the identity $1_G$. If $f\colon G_1\to G_2$ is an analytic map of analytic groups that preserves the identity, we denote by $df\colon\mathcal{L}(G_1)\to\mathcal{L}(G_2)$ its differential. If $f$ is also a group homomorphism, we have that $df$ is a Lie algebra homomorphism. 

When $K=\bbQ_p$, it is possible to describe $\mathcal{L}(G)$ in terms of open uniform pro-$p$ subgroups of $G$ as in \cite[Ch.~9]{DDMS}. We recall that a pro-$p$ group $H$ is \emph{uniform} if it is finitely generated, powerful and torsion-free. Moreover, a uniform pro-$p$ group $H$ can be endowed with a $\bbZ_p$-Lie algebra structure and we use the notation $L_H$ when we refer to $H$ with this structure. If $G$ is a $p$-adic analytic group, then $\mathcal{L}(G)\simeq \bbQ_p\otimes L_H$ for every open uniform pro-$p$ subgroup $H$. Moreover, if $f\colon G_1\to G_2$ is an homomorphism of $p$-adic analytic groups, then $df=1\otimes f_0$ where $f_0$ is the restriction of $f$ to a suitable open uniform pro-$p$ subgroup of $G_1$. Indeed, since $\mathcal L(G_i)$ does not depend on the choice of open uniform subgroup, one may take an open uniform subgroup $H_2$ of $G_2$ and choose an open uniform subgroup $H_1$ in $f^{-1}(H_2)$, in such a way that $f_0 = f|_{H_1}$ allows us to properly define $df$.

Given an automorphism $\alpha\in\mathrm{Aut}(G)$, we say that $\alpha$ is fixed-point-free on $X\subseteq G$ if $\alpha$ does not fix any $x\in X\smallsetminus{1}$. We say that $\alpha$ is \emph{uniformly} fixed-point-free if it is fixed-point-free on every open uniform pro-$p$ subgroup. Note that for $\alpha\in\mathrm{Aut}(G)$ the differential $d\alpha\colon \mathcal{L}(G)\to\mathcal{L}(G)$ is a $\bbQ_p$-Lie algebra automorphism.
\begin{lemma}\label{ufpf}
Let $G$ be a $p$-adic analytic group and $\alpha\in\mathrm{Aut}(G)$. The following are equivalent:
\begin{enumerate}
\item[(i)] If $\alpha$ is fixed-point-free on an open uniform pro-$p$ subgroup $H$;
\item[(ii)] $\alpha$ is uniformly fixed-point-free;
\item[(iii)] $d\alpha$ is fixed-point-free.
\end{enumerate}
\end{lemma}
\begin{proof}
Suppose that $\alpha$ fixes a non-trivial element of a uniform open pro-$p$ subgroup $K$ of $G$. Since $K$ is torsion-free and $H\cap K$ is of finite index in $K$, there is a natural number $n$ such that $1\neq x^n\in H\cap K$. This is a contradiction since $\alpha$ is fixed-point-free on $H$. Therefore $\alpha$ is uniformly fixed-point-free.

Now, let $d\alpha=1\otimes\alpha_0$ where $\alpha_0$ is the restriction of $\alpha$ to an open uniform pro-$p$ subgroup $H_0$. For $h\in H_0$ and $\lambda\in \bbQ_p$, we have $d\alpha(\lambda\otimes h)=\lambda\otimes \alpha(h)$. If $d\alpha$ is fixed-point-free, then $\alpha(h)\neq h$ for every $1\neq h\in H_0$. Therefore $\alpha$ is uniformly fixed-point-free by $(i)$. On the other hand, if $\alpha$ is uniformly fixed-point-free, then $\lambda\otimes\alpha(h)\neq \lambda\otimes h$ for every $\lambda\in\bbQ_p$ and $h\in H_0$. Therefore, $d\alpha$ is fixed-point-free.
\end{proof}
Given an element $g$ of a group $G$, let $c_g\in\rm{Aut}(G)$ be the conjugation by $g$. We are ready to give a characterisation of torsion probabilistic identities in compact $p$-adic analytic groups.
\begin{thm}\label{torsionpadic}
The following are equivalent for a compact $p$-adic analytic group $G$:
\begin{enumerate}
\item[(i)] $\mu(\rm{Tor}(G))>0$;
\item[(ii)] there is $t\in\rm{Tor}(G)$ such that $c_t$\ is a uniformly fixed-point-free automorphism of $G$;
\item[(iii)] there is $ t\in \rm{Tor}(G)$ such that $dc_t$ is a fixed-point-free automorphism of $\mathcal{L}(G)$.
\end{enumerate}
\end{thm}
\begin{proof}
By \cref{ufpf}, it is clear that $(ii)$ is equivalent to $(iii)$. Suppose $P(G,x^{m}) > 0$ for some $m\geq 1$. Then there exists an open uniform pro-$p$ subgroup $H\unlhd_o G$ and $t \in {\rm{Tor}}(G)$ such that every element in $Ht$ has order $m$ by \cref{thm:analytic}. If $u \in C_H(t)$, then
\[1=(ut)^{m} = uu^t\cdots u^{t^{m-1}} = u^{m}\]
implies $u=1$, since $H$ is torsion-free. So $c_t$ is a uniformly fixed-point-free automorphism of $G$.
Now, suppose that there is a torsion element $t$, $|t|=m$, such that the conjugation $c_t$ is uniformly fixed-point-free. Consider the analytic map 
	\begin{align*}
		f\colon G &\to G\\
		g &\mapsto g^tg^{-1}
	\end{align*} and note that every element of the form $tf(g)=t^{g^{-1}}$ has order $m$. Since the differential $df=dc_t-\mathrm{id}$ has trivial kernel by our assumption of $c_t$, we have that $df\colon\mathcal{L}(G)\to\mathcal{L}(G)$ is an isomorphism. By \cite[Pt.~2, Ch.~III.9, Theorem~2]{Ser65}, $f$ is a local isomorphism at $1_G$, therefore its image contains an open subgroup $H$. In particular, the word $x^m$ is satisfied on the coset $tH$.
\end{proof}

As an application, we give a different proof of \cref{cor:3.5} in the compact $p$-adic analytic case for torsion words. 
\begin{corollary}\label{cor:krekninpadic}
Let $G$ be a compact $p$-adic analytic group that is not virtually solvable. Then $\mu(\mathrm{Tor}(G))=0$.
\end{corollary}
\begin{proof}
Suppose that the set of torsion elements of $G$ has positive Haar measure. By \cref{torsionpadic} and \cite[Theorem 4.3.1]{Khu93}, $\mathcal{L}(G)$ is solvable since it admits a finite order fixed-point-free automorphism. In particular, if $H$ is an open uniform pro-$p$ subgroup of $G$, then $L_H$ is solvable as a Lie subalgebra of $\mathcal{L}(G)$. By \cite[Corollary 7.16]{DDMS}, $H$ is solvable. This is a contradiction since $G$ is not virtually solvable.
\end{proof}
In particular, for a $p$-adic analytic pro-$p$ group $G$, we can have $\mu(\mathrm{Tor}(G))>0$ only if $G$ is solvable.

In the virtually nilpotent case, we can give explicit examples of groups satisfying (and not satisfying) probabilistic torsion identities. We make essential use of \cref{thm:analytic}, that is, that a probabilistic torsion identity is also a coset identity.

\begin{example} %\begin{enumerate}
	%	\item 
        Let $U=\mathbb Z_3^2$ and consider $t \in \mathrm{GL}_2(\mathbb Z_3)$ to be the matrix
		\[\begin{pmatrix}
			0 & -1 \\ 1 & -1
		\end{pmatrix}.\]
		%\noindent 
        Notice that $t$ has order 3 and its minimal polynomial is $x^2+x+1$, whose roots are the primitive third roots of unity. Let $G = U \rtimes \langle t \rangle$. By standard calculations, one can check that the every element in the coset $Ut$ has order 3.
	%	\item 
    \end{example}
    \begin{example}
        Let $U = \mathbb Z_5^5$ and consider $t \in \mathrm{GL}_5(\mathbb Z_5)$ to be the matrix
		\[\begin{pmatrix}
			0 & 0 & 0 & -1 & 0 \\ 
			1 & 0 & 0 & -1 & 0 \\
			0 & 1 & 0 & -1 & 0 \\
			0 & 0 & 1 & -1 & 0 \\
			0 & 0 & 0 & 0 & 1 
		\end{pmatrix}.\]
		Let $G= U \rtimes \langle t \rangle$. If $G$ satisfies any probabilistic torsion identity, it must be $x^5$. Indeed, a non-trivial torsion element in $G$ must be of the form $ut^{j}$, with $u \in U\smallsetminus \{1\}$ and $1\leq j \leq 4$. Since $x=(ut^j)^5 = uu^tu^{t^2}u^{t^{3}}u^{t^4}$ is in $U$ and $U$ is torsion-free, we must have $x=1$. For any $H \leq_o G$, we can find in $H$ a normal subgroup of the form $5^{n_1}\mathbb Z_5 \times \cdots \times 5^{n_2} \mathbb Z_5$, with each $n_i$ a non-negative integer, so $Ht^{j}$ will have an element of infinite order for any $1\leq j \leq 4$. It follows that $x^5$ is not a coset identity for $G$ and thus not a probabilistic one. 
        \end{example}
        \begin{example}
        %\item 
        Let $p \geq 3$ and let $\zeta$ be a primitive $p$-th root of unity. We consider the ring $R = \bbZ_p[\zeta]$.
        Let 
        \[U = \left\{\begin{pmatrix} 1& x & z \\ 0 & 1 & y \\ 0&0&1 \end{pmatrix} \mid x,y,z \in R \right\}\]
        be the Heisenberg group over $R$ and we define 
        $t = \begin{pmatrix} 1 & 0 & 0 \\ 0 & \zeta & 0 \\ 0 & 0 & \zeta^2 \end{pmatrix}$.
        Then $G = \langle t, U \rangle$ splits as a semidirect product $G = U \rtimes \langle t \rangle$. Since conjugation with $t$ provides a fixed-point-free automorphism of $U$, $x^p$ is a coset identity for $G$.
	%\end{enumerate}
\end{example}

% Note that, if we restrict ourselves to virtually abelian $p$-adic analytic pro-$p$ groups, the first example above is in some sense the first non-trivial one. Precisely, if $G$ is a $p$-adic analytic pro-$p$ of dimension 1, then $G$ cannot have probabilistic torsion identities unless $p=2$, in which case we just get that $G$ has open subgroup isomorphic to $\mathbb Z_2 \rtimes C_2$, with the action by inversion. For dimensions greater than 1, there's a restriction on the dimens
We remark that all compact $p$-adic analytic groups $G$ which satisfy a torsion probabilistic identity of the form $x^p$ are virtually nilpotent. Indeed, by \cref{torsionpadic}, there exists an open uniform pro-$p$ subgroup $U$ of $G$ and $t \in \mathrm{Tor}(G)$ such that $C_U(t) = 1$ and consequently $dc_t$ is a fixed-point-free automorphism of $\mathcal L(G)$ of order $p$. Then $\mathcal L(G)$ is nilpotent by \cite[Theorem 3]{Z89} and thus so is $U$. 

As a final application of \cref{thm:analytic}, we show that solvable just-infinite $p$-adic analytic groups are either torsion-free or they satisfy a torsion probabilistic identity.
\begin{corollary}
Let $G$ be a just-infinite $p$-adic analytic pro-$p$ group. If $G \not\cong \bbZ_p$, then it either satisfies a torsion probabilistic identity or it is randomly free.
\end{corollary}
% \begin{proof}
% If $G$ is not randomly free, then it is solvable by \cref{thm:analytic}. In that case, it is virtually abelian and we can choose a maximal abelian open uniform subgroup $U$ of $G$. Being maximal, we have that $C_U(t) = 1$ for any $t \in {\rm{Tor}}(G)$ and thus $G$ satisfies a torsion probabilistic identity by \cref{torsionpadic}.
% \end{proof}

\begin{proof}
    If $G$ is not randomly free, then it is solvable by \cref{thm:analytic}. In that case, it is virtually abelian and we can choose a maximal abelian open normal subgroup $U$ of $G$.
    We have a short exact sequence
    \[1 \rightarrow U \rightarrow G \xrightarrow \pi B \rightarrow 1,\]
    where $U$ is infinite and $B$ is a finite $p$-group. 
    Observe that $U$ is torsion-free, since the torsion subgroup of $U$ is characteristic and hence normal in $G$; given that $G$ is just-infinite, this normal subgroup needs to be trivial. Similarly, for $G$ to be just-infinite, $U$ does not admit any $G$-invariant subgroups of infinite index.
    
    Since $G \not \cong \bbZ_p$, the group $B$ and thus its center $Z(B)$ are non-trivial. 
    Suppose first that $Z(B)$ acts trivially on $U$. Pick any non-trivial element $z \in B$ and consider $N=\pi^{-1}(\langle z \rangle)$. This is a normal subgroup of $G$ which contains $U$. As $N$ is a cyclic central extension of an abelian group, it is abelian; a contradiction to the maximality of $U$.
    So there exists $z \in Z(B)$ that acts non-trivially on $U$ and fixes no points of $U$; indeed, if it did, the subgroup of $U$ consisting of points fixed by $z$ would be a non-trivial $B$-invariant subgroup of infinite index. It follows from a standard calculation that $H^2(\langle z \rangle,U) = 0$ (see \cite[Theorem 6.2.2]{Weibel}), so in particular $U\rtimes \langle z \rangle$ is a subgroup of $G$ and $C_U(z)=1$. One concludes from \cref{torsionpadic} that $G$ satisfies a torsion probabilistic identity.
\end{proof}

%\bibliographystyle{alpha}
%\bibliography{References}

\end{document}